\documentclass{amsart}

\usepackage{amssymb}
\usepackage{amsthm}
\newtheorem{lemma}{Lemma}[section]
\newtheorem{remark}{Remark}[section]
\newtheorem{theorem}{Theorem}[section]
\newtheorem*{theorem*}{Theorem}
\newcommand{\R}{\mathbb{R}}
\newcommand{\N}{\mathbb{N}}
\newcommand{\be}{\begin{equation}}
\newcommand{\ee}{\end{equation}}
\newcommand{\g}{\gamma}

\numberwithin{equation}{section}

\begin{document}

\title[On weighted inequalities for fractional integrals \dots]{On weighted
inequalities for fractional integrals of radial functions}

\author{Pablo L.  De N\'apoli}
\address{Departamento de Matem\'atica, Facultad de Ciencias Exactas y Naturales,
Universidad de Buenos Aires, 1428 Buenos Aires, Argentina}
\email{pdenapo@dm.uba.ar}

\author{Irene Drelichman}
\address{Departamento de Matem\'atica, Facultad de Ciencias Exactas y Naturales,
Universidad de Buenos Aires, 1428 Buenos Aires, Argentina}
\email{irene@drelichman.com}

\author{Ricardo G. Dur\'an}
\address{Departamento de Matem\'atica, Facultad de Ciencias Exactas y Naturales,
Universidad de Buenos Aires, 1428 Buenos Aires, Argentina}
\email{rduran@dm.uba.ar}

\thanks{Supported by ANPCyT under grant PICT 01307 and by Universidad de
Buenos Aires under grants X070 and X837. The first and third authors are members of
CONICET,
Argentina.}


\begin{abstract}
We prove a weighted version of the Hardy-Littlewood-Sobolev inequality for
radially symmetric functions, and show that the range of admissible power
weights appearing in the classical inequality due to Stein and Weiss 
can be improved in this particular case.
\end{abstract}
\maketitle
\section{Introduction}
Consider the fractional integral operator
$$
(T_\gamma v)(x)= \int_{\R^n} \frac{v(y)}{|x-y|^\gamma} \; dy, \quad
0 <\gamma < n.
$$

Weighted estimates for this operator (also called weighted
Hardy-Littlewood-Sobolev inequalities) go back to G. H. Hardy and J. E.
Littlewood in the 1-dimensional case \cite{HL}, and were generalized to the
space $\mathbb{R}^n, n\ge 1$ by E. M. Stein and G. Weiss in the following
celebrated result:

\begin{theorem}\cite[Theorem B*]{SW}
\label{stein-weiss}
Let $n\ge 1$, $0<\gamma<n, 1<p<\infty, \alpha<\frac{n}{p'},
\beta<\frac{n}{q}, \alpha + \beta \ge 0$, and
$\frac{1}{q}=\frac{1}{p}+\frac{\gamma+\alpha+\beta}{n}-1$. If $p\le
q<\infty$, then the inequality
$$
\||x|^{-\beta}T_\gamma v\|_{L^q(\mathbb{R}^n)} \le C \||x|^\alpha
v\|_{L^p(\mathbb{R}^n)}
$$
holds for any $v\in L^p(\mathbb{R}^n, |x|^{p\alpha} dx)$, where $C$ is
independent of $v$.
\end{theorem}

Inequalities for the fractional integral with general weights were later studied
by several people, see for example \cite{SaW} and references therein. In
particular, it can be deduced from this theory (e.g., \cite[Theorem 1]{SaW})
that if we restrict ourselves to power weights, the previous theorem cannot be
improved in general.

However, if we reduce ourselves to radially symmetric functions, it is possible
to obtain a wider range of exponents for which the fractional integral is
continuous with power weights.  This is of particular interest for some
applications to partial differential equations (see, e.g. \cite{HK, V}).
More precisely, our main theorem is:

\begin{theorem}
\label{main-theorem}
Let $n\ge 1$, $0<\gamma<n, 1<p<\infty, \alpha<\frac{n}{p'},
\beta<\frac{n}{q}, \alpha + \beta \ge (n-1)(\frac{1}{q}-\frac{1}{p})$, and
$\frac{1}{q}=\frac{1}{p}+\frac{\gamma+\alpha+\beta}{n}-1$.  If $p\le
q<\infty$, then the inequality
$$
\||x|^{-\beta}T_\gamma v\|_{L^q(\mathbb{R}^n)} \le C \||x|^\alpha
v\|_{L^p(\mathbb{R}^n)}
$$
holds for all radially symmetric $v\in L^p(\mathbb{R}^n, |x|^{p\alpha} dx)$,
where $C$ is independent of $v$.
\end{theorem}

\begin{remark}
\label{rem1}
If $p=1$, then the result of Theorem \ref{main-theorem} holds for $\alpha +
\beta > (n-1)(\frac{1}{q}-\frac{1}{p})$ as may be seen from the proof of the
Theorem.
\end{remark}

\begin{remark}
\label{remark-alfabeta}
When $\gamma\le n-1$, the condition
$\frac{1}{q}=\frac{1}{p}+\frac{\gamma+\alpha+\beta}{n}-1$ automatically implies
$\alpha + \beta \ge (n-1)(\frac{1}{q}-\frac{1}{p})$.
\end{remark}

\begin{remark}
It is worth noting that if $n=1$ or $p=q$, Theorem \ref{main-theorem} gives the
same range of exponents as Theorem \ref{stein-weiss}.
\end{remark}

Our method of proof can also  be used, with slight modifications, to re-obtain
the general case of Stein and Weiss' Theorem. Since this result is not new and
the main ideas needed will already appear in our proof of the radial case, we
leave the details to the reader.

Previous results in the direction of Theorem \ref{main-theorem} include the
works of  M. C. Vilela, who made a proof for the case $p<q$ and $\beta=0$ in
\cite[Lemma 4]{V}; the work of G. Gasper, K. Stempak and W. Trebels \cite[Theorem 3.1]{Ga}, who proved a  fractional integration theorem in the context of Laguerre expansions which in the particular case of radial functions in $\mathbb{R}^n$ gives Theorem  \ref{main-theorem}  for $\alpha+\beta\ge 0$;
and the work of K. Hidano and Y. Kurokawa \cite[Theorem
2.1]{HK}, who proved Theorem  \ref{main-theorem}  for $p< q$ under the
stronger condition $\beta<\frac{1}{q}$. Notice that  this restriction,
together with the additional assumptions on $\alpha$ and $\beta$ implies
$n-1<\gamma<n$, whereas our conditions on $\alpha$ and $\beta$ allow for
$0<\gamma<n$, which is the natural range of $\gamma$'s for the fractional
integral.  This is because the proof  of Hidano and Kurokawa reduces to the
1-dimensional case of the Stein-Weiss theorem, while our method of proof is
completely different. Morover, our proof is simpler than that in  \cite{HK},
particularly when $n=2$.

The rest of the paper is organized as follows: in Section 2 we  recall some
definitions and preliminary results that will be needed in this paper. In
Section 3 we prove Theorem \ref{main-theorem} in the case $n=1$. As we have
already pointed out, in this case our range of weights coincides with that of
Stein and Weiss (and Hardy and Littlewood) and, therefore, the assumption that
$v$ be radially symmetric (i.e., even) is unnecessary. Although this result is
not new, we have included it because the proof we provide is very simple and
uses some of the ideas that we will use to prove the general theorem. Section 4
is devoted to the proof of Theorem  \ref{main-theorem} in the general case, and
we show, by means of an example when $n=3$, that the condition on $\alpha+\beta$
is sharp. Finally, in Section 5 we use Theorem  \ref{main-theorem}  to obtain a
weighted imbedding theorem for radially symmetric functions.

\section{Preliminaries}
Let $X$ be a measure space and $\mu$ be a positive measure on $X$. Recall that
if $f$ is a measurable function, its distribution function $d_f$ on $[0,\infty)$
is defined as
$$
d_f(\alpha) = \mu ( \{ x\in X: |f(x)|>\alpha \}).
$$
For $0<p<\infty$ the space weak-$L^p(X,\mu)$ is defined as the set of all
$\mu$-measurable functions $f$ such that $\|f\|_{L^{p,\infty}}$ is finite, where
$$
\|f\|_{L^{p,\infty}} = \inf \{ÊC>0 : d_f(\alpha)\le (C\alpha^{-1})^p \mbox{
for all } \alpha>0 \}.
$$

If $G$ is a locally compact group, then $G$ posseses a Haar measure, that is, a
positive Borel measure $\mu$ which is left invariant (i.e. $\mu(At)=\mu(A)$
whenever $t\in G$ and $A\subseteq G$ is measurable)   and nonzero on nonempty
open sets. In particular, if $G=\mathbb{R}^*:=\mathbb{R}-\{0\}$ (with
multiplicative structure),  then $\mu= \frac{dx}{|x|}$, and if $G=
\mathbb{R}^+$, then $\mu=\frac{dx}{x}$.

The convolution of two functions $f,g\in L^1(G)$  is defined as:
$$
(f*g)(x) = \int_G f(y) g(y^{-1}x) \, d\mu(y)
$$
where $y^{-1}$ denotes the inverse of $y$ in the group $G$.

With these definitions in mind, we are ready to recall the following improved
version of Young's inequality that will be needed in what follows:

\begin{theorem}\cite[Theorem 1.4.24]{G}
\label{young}
Let $G$ be a locally compact group with left Haar measure $\mu$ that satisfies
$\mu(A)=\mu(A^{-1})$ for all measurable $A\subseteq G$. Let
$1<p,q,s<\infty$ satisfy
$$
\frac{1}{q}+1 = \frac{1}{p}+\frac{1}{s}.
$$
Then, there exists a constant $B_{pqs}>0$ such that for all $f\in L^p(G,
\mu)$ and $g\in L^{s,\infty}(G, \mu)$ we have
\be
\|f * g\|_{L^q(G, \mu)}\le B_{pqr} \|g\|_{L^{s,\infty}(G, \mu)} \|f\|_{L^p(G,
\mu)}.
\label{young2}
\ee
\end{theorem}

\begin{remark}
Notice that if $p=1$, and $g\in L^s(G,\mu)$, then  we can replace (\ref{young2})
by the classical Young's inequality, to obtain
$$
\|f * g\|_{L^q(G, \mu)}\le B_{pqr} \|g\|_{L^{s}(G, \mu)} \|f\|_{L^p(G, \mu)}.
$$
This can be used to prove the extension to the case $p=1$ of Theorem
\ref{main-theorem} (see Remark \ref{rem1}).
\end{remark}

\section{The 1-dimensional case}

Recall that we want to prove
\begin{equation}
\label{dim-1}
\| |x|^{-\beta} T_\gamma f\|_{L^q(\mathbb{R})} \le C \| f |x|^{\alpha}
\|_{L^p(\mathbb{R})}
\end{equation}
The key point in our proof is to write the above inequality as a convolution
inequality in the group $\mathbb{R}^*$ with the corresponding Haar measure $\mu=
\frac{dx}{|x|}$. Indeed, inequality \eqref{dim-1}
can be rewritten as
$$
\| |x|^{-\beta+\frac{1}{q}} T_{\gamma} f \|_{L^q(\mu)} \le C
\||x|^{\alpha+\frac{1}{p}} f \|_{L^p(\mu)}
$$

Now,
$$
|x|^{-\beta+\frac{1}{q}} T_\gamma f(x) = \int_{-\infty}^\infty
\frac{|x|^{-\beta+\frac{1}{q}} f(y)
|y|^{\alpha+\frac{1}{p}}}{|y|^{\gamma-1+\alpha+\frac{1}{p}}
|1-\frac{x}{y}|^{\gamma}} \frac{dy}{|y|} = ( h * g)(x)
$$
where $h(x)=f(x) |x|^{\alpha+\frac{1}{p}}, g(x)=
\frac{|x|^{-\beta+\frac{1}{q}}}{|1-x|^\gamma}$,
and we have used that $\gamma-1+\alpha+\frac{1}{p}= -\beta+\frac{1}{q}$.
Using Theorem \ref{young} we obtain

$$
\| |x|^{-\beta+\frac{1}{q}} T_{\gamma} f \|_{L^q(\mu)} \le C
\||x|^{\alpha+\frac{1}{p}} f \|_{L^p(\mu)} \|g(x)\|_{L^{s,\infty}(\mu)}
$$

where
$$
\frac{1}{q}=\frac{1}{p}+\frac{1}{s}-1
$$

Therefore, it suffices to check that $\|g(x)\|_{L^{s,\infty}(\mu)} < \infty$.
For this purpose, consider $\varphi \in C^{\infty}(\mathbb{R})$, supported in
$[\frac12, \frac32]$ and such that $0\le \varphi\le 1$ and $\varphi\equiv 1$ in
$(\frac34, \frac54)$. We split $g = \varphi g + (1-\varphi) g := g_1 + g_2$.

Clearly, $g_2 \in L^s(\mu)$, since the integrability condition at the origin for
$|g_2|^t$(with respect to the measure $\mu$) is $\beta<\frac{1}{q}$, and the
integrability condition when  $x \to\infty$ is $\frac{1}{q}-\beta-\gamma<0$,
which, under our assumptions on the exponents, is equivalent to
$\alpha<\frac{1}{p'}$.

Therefore,
$$
\mu(\{g_1 + g_2 > \lambda \}) \le \mu\left(\left\{g_1 > \frac{\lambda}{2}
\right\}\right) + \left( \frac{\|g_2\|_{L^s(\mu)}}{\lambda}\right)^s \le
\mu\left(\left\{g_1 > \frac{\lambda}{2} \right\}\right) + \frac{C}{\lambda^s}
$$
but,
$$
\mu\left(\left\{g_1 > \frac{\lambda}{2} \right\}\right) \le \mu\left(\left\{
\frac{C}{|1-x|^\gamma} > \lambda \right\}\right) = \mu\left(\left\{
\frac{C}{\lambda^\frac{1}{\gamma}} >|x-1| \right\}\right)\le
\frac{C}{\lambda^\frac{1}{\gamma}}\le \frac{C}{\lambda^s}
$$
as long as  $s\gamma \le 1$, that is $\gamma\le 1+\frac{1}{q}-\frac{1}{p}$,
which is equivalent to $\alpha+\beta \ge 0$. Hence, $g\in L^{s,\infty}(\mu)$ and
this concludes the proof.

\section{Proof of the weighted HLS theorem for radial functions}

In this Section we prove Theorem \ref{main-theorem}. The  main idea, as in the
one-dimensional case will be to write the fractional integral operator acting on
a radial function as a convolution in the multiplicative group $\mathbb{R}^+$
with Haar measure $\mu=\frac{dx}{x}$. For this purpose, we shall need the
following lemma.

\begin{lemma}
Let $x \in S^{n-1}=\{ x \in \R^n : |x|=1\}$ and consider an integral of the form:
$$ I(x) = \int_{S^{n-1}} f(x \cdot y) \; dy $$
(the integral is taken with respect to the
surface measure on the sphere), where $f:[-1,1] \to \R, f \in
L^1([-1,1],(1-t^2)^{(n-3)/2})$. Then, $I(x)$ is a constant independent of $x$
and moreover
$$ I(x)= \omega_{n-2} \int_{-1}^1f(t)  (1-t^2)^{\frac{n-3}{2}} dt $$
\label{lemma-integral-on-the-sphere}
where $\omega_{n-2}$ denotes the area of $S^{n-2}$.
\end{lemma}

\begin{proof}
First, observe that $I(x)$ is constant for all $x\in S^{n-1}$. Indeed, given
$\tilde x \in S^{n-1}$, there exists a rotation $R\in O(n)$ such that $\tilde x=
Rx$ and, therefore,
$$
I(\tilde x)= \int_{S^{n-1}} f(\tilde x \cdot y) \, dy =  \int_{S^{n-1}} f(Rx
\cdot y) \, dy = \int_{S^{n-1}} f(x \cdot R^{-1} y) \, dy = I(x).
$$
So, taking $x=e_n$, it suffices to compute $I(e_n)= \int_{S^{n-1}} f(y_n)\, dy$.
To this end, we split the integral in two and consider first the integral on the
upper-half sphere $(S^{n-1})^+$. Since $(S^{n-1})^+$ is the graph of the
function $g: \{ x\in \mathbb{R}^{n-1}: |x| <1\} \to (S^{n-1})^+, g(x)=
\sqrt{1-|x|^2}$, we obtain
$$
\int_{(S^{n-1})^+} f(y_n) \, dy = \int_{\{|x|<1\}} f(\sqrt{1-|x|^2})
\frac{1}{\sqrt{1-|x|^2}} \, dx
$$
using polar coordinates, this is
$$
\int_{S^{n-2}} \int_0^1 f(\sqrt{1-r^2}) \frac{1}{\sqrt{1-r^2}} r^{n-2} \, dr \,
dy = \omega_{n-2} \int_0^1 f(t) (1-t^2)^{\frac{n-3}{2}} \, dt.
$$
Analogously, one obtains
$$
\int_{(S^{n-1})^-} f(y_n) \, dy = \omega_{n-2} \int_{-1}^0 f(t)
(1-t^2)^{\frac{n-3}{2}} \, dt.
$$
This completes the proof.
\end{proof}

Now we can proceed to the proof of our main theorem.

Using polar coordinates,
$$ y = ry^\prime, r=|y|, y^\prime \in S^{n-1} $$
$$ x = \rho x^\prime, \rho = |x|, x^\prime \in S^{n-1} $$
and the identity
$$|x-y|^2 = |x|^2 - 2 |x||y| x^\prime \cdot y^\prime + |y|^2 $$
we write the fractional integral of a radial function $v(x)=v_0(|x|)$ as
$$T_\gamma v(x) = \int_0^\infty \int_{S^{n-1}}
\frac{v_0(r) r^{n-1} dr dy^\prime}
{(r^2-2r\rho x^\prime y^\prime + \rho^2)^{\gamma/2}} $$
Using lemma \ref{lemma-integral-on-the-sphere}, we have
that:
$$T_\gamma v(x) = \omega_{n-2} \int_0^\infty v_0(r) r^{n-1} \left\{
\int_{-1}^1 \frac{(1-t^2)^{(n-3)/2} }{(\rho^2-2 \rho r t + r^2)^{\gamma/2}}
\; dt \right\} dr $$
Now, we may write the inner integral as:
$$ \int_{-1}^1 \frac{(1-t^2)^{(n-3)/2} }{(\rho^2-2 \rho r t + r^2)^{\gamma/2}}
\; dt =\int_{-1}^1 \frac{(1-t^2)^{(n-3)/2} }{\rho^\gamma
\left[1-2 \left(\frac{r}{\rho}\right) t + \left(\frac{r}{\rho}\right)
^2\right]^{\gamma/2}}
\; dt  $$
Therefore,
$$ T_\gamma v(x)=  \omega_{n-2} \int_0^\infty v_0(r) r^n
 \frac{1}{\rho^\gamma} I_{\gamma,k} \left(\frac{r}{\rho}\right) \frac{dr}{r}
\label{fracint-radial} $$
where $k= \frac{n-3}{2} $, and, for $a\ge 0$
$$ I_{\gamma,k} (a)= \int_{-1}^1  \frac{(1-t^2)^{k} }{(1-2 a t +
a^2)^{\gamma/2}} \;dt $$
Notice that the denominator of this integral vanishes if $a=1$ and $t=1$ only.
Therefore, $I_{\gamma,k}(a)$ is well defined and is a continuous function for $a
\neq 1$.

This formula shows in a explicit way that $T_\gamma v$ is a radial function, and
can be therefore thought of as a function of $\rho$. Furtheremore, we observe
that as consequence of this formula,  $\rho^{\frac{n}{q}-\beta} T_\gamma v$
has the structure of a convolution on the multiplicative group $\R^+$:
$$
\rho^{\frac{n}{q}-\beta} T_\gamma v(x)
= \omega_{n-2} \int_0^\infty v_0(r) r^{n-\gamma+\frac{n}{q}}
\frac{\rho^{\frac{n}{q}-\beta}} {r^{\frac{n}{q}-\beta}}
I_{\gamma,k} \left(\frac{r}{\rho}\right) \frac{dr}{r}
$$
$$
= \omega_{n-2} \; (v_0 r^{n-\gamma-\frac{n}{q}-\beta}) *
(r^{\frac{n}{q}-\beta} I_{\gamma,k}(r))
 $$

Hence, using Theorem \ref{young} we get that
$$
 \| |x|^{-\beta} T_\gamma v \|_{L^q(\R^n)}
= \left( \omega_{n-1} \int_0^\infty |T_\gamma v(\rho)|^{q} \rho^{n-\beta q}
\; \frac{d\rho}{\rho} \right)^{1/q}
= \omega_{n-1}^{1/q} \; \| T_\gamma v(\rho) \rho^{\frac{n}{q}-\beta}
\|_{L^q(\mu)}
$$
$$
\leq \omega_{n-1}^{1/q} \omega_{n-2} \| v_0(r) r^{n-\gamma+\frac{n}{q}-\beta}
\|_{L^p(\mu)} \;
\| r^{\frac{n}{q}-\beta} I_{\gamma,k}(r) \|_{L^{s,\infty} (\mu)}
$$
provided that:
\be  \quad \frac{1}{p}+\frac{1}{s}-1 = \frac{1}{q}
\label{Young-conditions}
\ee

Using polar coordinates once again:
$$ \omega_{n-1}^{1/p} \; \| v_0(r) r^{n-\gamma+\frac{n}{q}-\beta} \|_{L^p(\mu)}
= \omega_{n-1}^{1/p} \; \left( \int_0^\infty |v_0(r)|^{p}
r^{\left(n-\gamma+\frac{n}{q}-\beta\right)p-n}
r^n \frac{dr}{r} \right)^{1/p}  $$
$$ = \| v_0 |x|^{n-\gamma+\frac{n}{q}-\beta-\frac{n}{p}} \|_{L^p(\R^n)} $$

But, by the conditions of our theorem,
$$
n-\gamma+\frac{n}{q}-\beta-\frac{n}{p} = \alpha.
$$
Therefore, it suffices to prove that

\be
\| r^{\frac{n}{q}-\beta} I_{\gamma,k}(r) \|_{L^{s,\infty} (\mu)}
 < +\infty. \label{integrability} \ee

For this purpose, consider $\varphi \in C^{\infty}(\mathbb{R})$, supported in
$[\frac12, \frac32]$ and such that $0\le \varphi\le 1$ and $\varphi\equiv 1$ in
$(\frac34, \frac54)$. We split $ r^{\frac{n}{q}-\beta} I_{\gamma,k}= \varphi
r^{\frac{n}{q}-\beta} I_{\gamma, k} + (1-\varphi) r^{\frac{n}{q}-\beta}
I_{\gamma,k} := g_1 + g_2$.

We claim that $g_2 \in L^s(\mu)$. Indeed, since $I_{\gamma,k}(r)$ is a
continuous function for $r\neq 1$, to analyze the behavior (concerning
integrability) of $g_2$ it suffices to consider the behavior  of
$r^{(\frac{n}{q}-\beta)s} |I_{\gamma,k}(r)|^s$ at
$r=0$, and when $r \to +\infty$.

Since $I_{\gamma,k}(r)$ has no singularity at $r=0$ ($I_{\gamma,k}(0)$ is
finite) the local integrability condition at $r=0$ is $\beta < \frac{n}{q} $.

When $r \to +\infty$, we observe that
$$
I_{\gamma,k}(r) =   \frac{1}{r^\gamma} \int_{-1}^1
\frac{(1-t^2)^{k} }{(r^{-2}-2 r^{-1} t + 1)^{\gamma/2}} \;dt
 $$
and using the bounded convergence theorem, we deduce that
$$
 I_{\gamma,k}(r) \sim \frac{C_k}{r^\gamma} \quad \hbox{as} \; r \to
+\infty \quad  (\hbox{with}\; C_k= \int_{-1}^1 (1-t^2)^{k} dt )
$$
It follows that the integrability condition at infinity is $ \frac{n}{q}-\beta -
\gamma   <0$, which, under our conditions on the exponents, is equivalent to
$\alpha < \frac{n}{p'}$.

We proceed now to $g_1$. To analyze its behavior near $r=1$, we shall need the
following lemma:

\begin{lemma}
\label{lemma42}
For $a \sim 1$ and $k \in \N_0$ or $k=m - \frac{1}{2}$ with $m \in \N_0$,
we have that
$$ |I_{\gamma,k}(a)| \leq
 \left\{
\begin{array}{lclcc}
C_{k,\gamma} & \mbox{if} &  \gamma<n-1 \\
C_{k,\gamma} \log\frac{1}{|1-a|}& \mbox{if} &  \gamma=n-1 \\
C_{k,\gamma} |1-a|^{-\gamma+2k+2} & \mbox{if} & n-1<\gamma<n\\
\end{array}
\right.
 $$
\end{lemma}
\begin{proof}
Assume first that $k\in\N_0$ and $-\frac{\gamma}{2}+k > -1$ (that is,
$0<\gamma<n-1$). Then,

$$
I_{\g,k}(1)\sim\int_{-1}^1 \frac{(1-t^2)^k}{(2-2t)^{\frac{\gamma}{2}}} \, dt
\sim C \int_{-1}^1 \frac{(1-t)^k}{(1-t)^{\frac{\gamma}{2}}} \, dt
$$
Therefore,  $I_{k, \gamma}$ is bounded.

If $-\frac{\gamma}{2}+k =-1$, (that is, $\gamma=n-1$) then

$$
I_{\g,k}(a)\sim\int_{-1}^1 (1-t^2)^k
\frac{d^k}{dt^k}\left\{(1-2at+a^2)^{-\frac{\g}{2}+k}\right\}\, dt
$$
Integrating by parts $k$ times (the boundary terms vanish),
$$
I_{\g,k}(a)\sim\left|\int_{-1}^1 \frac{d^k}{dt^k}\left\{(1-t^2)^k\right\}
(1-2at+a^2)^{-\frac{\g}{2}+k}\, dt\right|
$$
But $\frac{d^k}{dt^k}\left\{(1-t^2)^k\right\}$ is a polynomial of degree $k$
and therefore is bounded in $[-1,1]$ (in fact, it is up to a constant the
classical Legendre polynomial). Therefore,

$$
I_{k, \gamma}(a) \sim 2a \log\left(\frac{1+a}{1-a}\right)^2 \le C \log
\frac{1}{|1-a|}.
$$

Finally, if  $-\frac{\gamma}{2}+k <-1$ (that is, $n-1<\gamma<n$), then
integrating by parts as before,
$$
I_{\g,k}(a)\le C_k\int_{-1}^{1}(1-2at+a^2)^{-\frac{\g}{2}+k}\, dt
$$

$$
I_{k, \gamma}(a) \sim (1-2at+a^2)^{-\frac{\g}{2}+k+1}|_{t=-1}^{t=1}
\le C_{k,\g} |1-a|^{-\g+2k+2}
$$

This finishes the proof if $k\in \mathbb{N}_0$(i.e. if $n$  is odd).

We proceed now to the case $k=m+\frac12$, $m\in \mathbb{N}_0$. Assume first that
$-\frac{\gamma}{2}+m+1<-1$. Then,
$$
I_{\g,k}(a)=\int_{-1}^1 (1-t^2)^k (1-2at+a^2)^{-\frac{\g}{2}}\, dt
$$
$$
=\int_{-1}^1 (1-t^2)^{\frac12 m}(1-2at+a^2)^{-\frac{\g}{4}}
(1-t^2)^{\frac12(m+1)} (1-2at+a^2)^{-\frac{\g}{4}}\, dt
$$
and applying the Cauchy-Schwarz inequality we get that
$$
I_{\g,k}(a) \leq I_{\g,m}(a)^{\frac12} I_{\g,m+1}(a)^{\frac12}
$$
and,  using the bound for the case in which $k$ is an integer for $I_{\g,m}(a)$
and  $I_{\g,m+1}(a)$, we conclude that, $I_{k, \gamma}(a) \le C
|1-a|^{-\gamma+2m+3}= C |1-a|^{-\gamma+2k+2}$

If $-\frac{\gamma}{2}+k\le -1$, then
$$
|I'_{\gamma, k}(a)|\le C \int_{-1}^1
\frac{(1-t^2)}{(1-2at+a^2)^{\frac{\gamma}{2}+1}} \, dt = C I_{\gamma+2,k}(a)
$$
But, $-\frac{\gamma+2}{2}+k+\frac{1}{2}= -\frac{\gamma}{2}+k-\frac12 < -1$,
therefore, $I_{\gamma+2,k}$ can be bounded as in the previous case to obtain
$I_{\gamma+2,k}(a)\le C |1-a|^{-\frac{\gamma}{2}+k}$. Using this bound, and the
fact that $-\frac{\gamma}{2}+k<-1$, we obtain
$$
I_{\gamma,k}(a) = \int_0^a I'_{\gamma,k}(s) \, ds \le C
(1-s)^{-\frac{\gamma}{2}+k+1} \big|_0^a \le C |1-a|^{-\frac{\gamma}{2}+k+1}.
$$

Finally, if $-\frac{\gamma}{2}+k=-1$,
$$
I_{\gamma,k}(a)\le C \int_0^a \frac{1}{1-s} \, ds = C \log \frac{1}{|1-a|}.
$$

It remains to check the case $k=-\frac12$
$$
I_{\gamma, -\frac{1}{2}}(a) = \int_{-1}^0
\frac{(1-t^2)^{-\frac{1}{2}}}{(1-2at+a^2)^\frac{\gamma}{2}} \, dt + \int_0^1
\frac{(1-t^2)^{-\frac{1}{2}}}{(1-2at+a^2)^\frac{\gamma}{2}} \, dt = I  + II
$$

Since $\gamma>0$,
$$
I \le \int_{-1}^0 \frac{dt}{(1+t)^{\frac12}}  = 2
$$
$$
II \le \int_0^1 \frac{(1-t)^{-\frac12}}{(1-2at+a^2)^{\frac{\gamma}{2}}} \, dt =
-2 \int_0^1 \frac{\frac{d}{dt}[(1-t)^{\frac12}]}{(1-2at+a^2)^{\frac{\gamma}{2}}}
\, dt
$$
$$
\le 2a\gamma \int_0^1 \frac{(1-t^2)^{\frac12}}{(1-2at+a^2)^{\frac{\gamma}{2}+1}}
\, dt \le C I_{\gamma+2,\frac12}
$$
\end{proof}

Now we can go back to the study of $g_1$. We shall split the proof into three
cases, depending on whether $\gamma$ is less than, equal to or greater than
$n-1$.
\begin{enumerate}
\item[i.] Assume first that $0< \gamma<n-1$. Then
$|r|^{(-\beta+\frac{n}{q})s} |I_{\gamma,k}|^s$ is bounded when $r\sim 1$, and,
therefore, $\|g_1\|_{L^s(\mu)}< +\infty$.

\item[ii.] Consider now the case $\gamma=n-1$. Since in this case
$$
|I_{\gamma,k}(r)| \le C \log\frac{1}{|1-r|},
$$
we conclude, as before, that $\|g_1\|_{L^s(\mu)}< +\infty$

\item[iii.] Finally, we have to consider the case $n-1<\gamma<n$. In this
case,
$$
|I_{\gamma,k}(r)| \leq C |r-1|^{-\gamma+2k+2} = C |1-r|^{-\gamma+n-1}
 $$
Therefore,
$$
\mu\left(\left\{g_1 > \frac{\lambda}{2} \right\}\right) \le \mu\left(\left\{
\frac{C}{|1-x|^{\gamma-n+1}} > \lambda \right\}\right) = \mu\left(\left\{
\frac{C}{\lambda^\frac{1}{\gamma-n+1}} >|x-1| \right\}\right)
$$
$$
\le \frac{C}{\lambda^\frac{1}{\gamma-n+1}}\le \frac{C}{\lambda^s}
$$
as long as  $s(\gamma-n+1) \le 1$, which is equivalent to $\alpha+\beta \ge
(n-1)(\frac{1}{p}-\frac{1}{q})$.
 Therefore, $\|g_1\|_{L^{s,\infty}(\mu)}< +\infty$.
\end{enumerate}

\begin{flushright}
$\Box$
\end{flushright}

\begin{remark} The following example shows that for $n=3$ the condition
$\alpha+\beta \ge (n-1)(\frac{1}{q}-\frac{1}{p})$ is necessary.

Assume that  $\alpha+\beta <(n-1)(\frac{1}{q}-\frac{1}{p})$. Then, by Remark
\ref{remark-alfabeta}, $\gamma>n-1$.

Since $\frac{1}{q}=\frac{1}{p}+\frac{1}{s}-1$, we obtain
$\gamma-n+1>\frac{1}{s}$ and, therefore, by Lemma \ref{lemma42}, for $n=3$ and $r \sim 1$,
$I_{\gamma,k}(r)\sim \frac{1}{|1-r|^{\frac{1}{s}+\varepsilon}}$ for some
$\varepsilon>0$.

Fix $\eta$ such that $\eta p >1$ and let
$$
 f(r)=\frac{\chi_{[\frac12,\frac32]}(r)}{|r-1|^{\frac{1}{p}} \log
(\frac{1}{|r-1|})^\eta}
$$

Then $f\in L^p(\mu)$ and, for $r>1$,
$$
(I_{\gamma,k}* f)(r) \ge  \int_r^\frac32
\frac{t^{\frac{1}{s}+\varepsilon}}{t^{\frac{1}{s}+\varepsilon}
|\frac{r}{t}-1|^{\frac{1}{s}+\varepsilon} |t-1|^{\frac{1}{p}} \log
(\frac{1}{|t-1|})^\eta} \, dt
$$
$$
\ge \int_r^\frac32 \frac{1}{(t-r)^{\frac{1}{s}
+\varepsilon}(t-1)^\frac{1}{p} (\log\frac{1}{|r-1|})^\eta} \, dy \ge \frac{1}{(\log
\frac{1}{|r-1|})^\eta} \int_r^\frac32
\frac{dy}{(t-1)^{\frac{1}{s}+\frac{1}{p}+\varepsilon}}
$$
$$
\sim \frac{1}{(\log\frac{1}{|r-1|})^\eta |r-1|^{\frac{1}{q}+\varepsilon}} \not\in L^q
$$

Recall now that for a radial function,
$$
\rho^{\frac{n}{q}-\beta} T_\gamma f_0(\rho) = f_0 r^{\frac{n}{p}+\alpha} *
r^{\frac{n}{q}-\beta} I_{\gamma,k}(r)
$$

Therefore, defining $f_0= f(|x|) |x|^{-\frac{n}{p}-\alpha}$ we have,   
$\|f _0|x|^\alpha\| _{L^p}< \infty$ but $T_\gamma f |x|^{-\beta} \not \in L^q$.
\end{remark}

\section{An application to weighted imbedding theorems}

Consider the fractional order Sobolev space
$$
H^s(\R^n)= \{u \in L^2(\R^n) : (-\Delta)^{s/2}u \in L^2(\mathbb{R}^n) \}  \,
(s\ge 0)
 $$

As an application of our main theorem, we will prove a  weighted imbedding
theorem for  $H_{rad}^s (\R^n)$, the subspace of radially symmetric functions of
$H^s(\R^n)$.

\begin{theorem}
Let $0<s<\frac{n}{2}$,\,\, $2<q<  2^*_c := \frac{2(n+c)}{n-2s}$.
Then, we have the compact imbedding
$$
H^s_{rad} (\mathbb{R}^n) \subset L^q(\mathbb{R}^n,|x|^c dx)
$$
provided that $ -2s < c < \frac{(n-1)(q-2)}{2}$.
\end{theorem}

\begin{remark}
The case $s=1$ of this lemma was already proved in the work of W. Rother
\cite{Rother}, while the general case was already proved in a completely
different way in our work \cite{DDD}.
\end{remark}
\begin{remark}
The unweighted case $c=0$ gives the classical Sobolev
imbedding (in the case of radially symmetric functions). In that case, the
compactness of the imbedding $H^{s}_{rad}(\R^n) \subset L^q(\R^n)$ under the
conditions $0<s<\frac{n}{2}$ and $2<q< \frac{2n}{n-2s}$ was proved
by P. L. Lions \cite{Lions}.
\end{remark}

\begin{proof}
Let $u\in H^s_{rad}(\mathbb{R}^n)$.  Then, $f:= (-\Delta)^{s/2}u \in L^2$, and,
recalling the relation between the negative powers of the Laplacian and the
fractional integral (see, e.g., \cite[Chapter V]{St}), we obtain
$$T_{n-s}f = C (-\Delta)^{-s/2}f = C u.$$

 Then, it follows from Theorem \ref{main-theorem} that
$$
\| |x|^{\frac{c}{q}} u \|_{L^{2^*_c}(\mathbb{R}^n)} = C \| |x|^{\frac{c}{q}}
T_{n-s} f \|_{L^{2^*_c}(\mathbb{R}^n)} \le C \|f \|_{L^2(\mathbb{R}^n)} \le C \|
u\|_{H^s(\mathbb{R}^n)}
$$

Therefore, writing $q= 2 \nu + (1-\nu) 2^*_c$, and using H\"older's
inequality, we obtain
$$
\||x|^{\frac{c}{q}} u\|_{L^q(\mathbb{R}^n)} \le \| |x|^{\frac{c}{q}}
u\|_{L^{2^*_c}(\mathbb{R}^n)}^\nu \|u\|_{L^2(\mathbb{R}^n)}^{1-\nu} \le C
\|u\|_{H^s(\mathbb{R}^n)}
$$

It remains to prove that the imbedding $H^s_{rad} (\mathbb{R}^n)
\subset L^q(\mathbb{R}^n,|x|^c dx)$ is compact. The proof can be made in the
same way as that in \cite[Theorem 2.1]{DDD}. Indeed,
it suffices to
show that if $u_n \to 0 $ weakly in $H_{rad}^s(\R^n)$, then $u_n \to 0$
strongly in $L^q(\R^n,|x|^c \; dx)$. Since
$$2<q<2^*_{c}=\frac{2(n+c)}{n-2s}$$ by hypothesis, it is possible
to choose $r$ and $\tilde{q}$ so that $2<r<q<\tilde{q}<2^*_{c}$.
We write $q = \theta r + (1-\theta) \tilde{q}$ with $\theta \in
(0,1)$ and, using H\"older's inequality, we have that

\be \int_{\R^n} |x|^{c} |u_n|^q \, dx \le \left( \int_{\R^n}
|u_n|^r \, dx \right)^{\theta} \left( \int_{\R^n} |x|^{\tilde{c}}
|u_n|^{\tilde{q}} \, dx \right)^{1-\theta}
\label{ineq-interpolacion} \ee where
$\tilde{c}=\frac{c}{1-\theta}$. By choosing $r$ close enough to
$2$ (hence making $\theta$ small), we can fulfill the conditions
$$ \tilde{q} <\frac{2(n+\tilde{c})}{n-2s},\quad -2s < \tilde{c} <
\frac{(n-1)(\tilde{q}-2)}{2}. $$
Therefore, by the imbedding that we have already established:
$$ \left( \int_{\R^n} |x|^{\tilde{c}} |u_n|^{\tilde{q}} \, dx
\right)^{1/\tilde{q}}
\leq C \| u_n \|_{H^s} \leq C^\prime.$$

Since the imbedding $H_{rad}^s(\R^n) \subset L^{r}(\R^n)$ is compact by Lions
theorem \cite{Lions}, we
have that $u_n \to 0$ in $L^{r}(\R^n)$. From (\ref{ineq-interpolacion}) we
conclude
that $u_n \to 0$ strongly in $L^q(\R^n,|x|^c\; dx)$, which shows that the
imbedding in our
theorem is also compact. This concludes the proof.

\end{proof}

\end{document}